\documentclass{amsart}
\usepackage{amssymb,stmaryrd,faktor}

\newtheorem{thm}{Theorem}[section]
\newtheorem*{thma}{Theorem~A}
\newtheorem*{thmb}{Theorem~B}
\newtheorem*{thmap}{Theorem~A'}

\newtheorem{lemma}[thm]{Lemma}
\newtheorem{cor}[thm]{Corollary}
\newtheorem{claim}{Claim}[thm]

\theoremstyle{definition}
\newtheorem{defn}[thm]{Definition}
\newtheorem{notation}[thm]{Notation}

\theoremstyle{remark}
\newtheorem{remark}[thm]{Remark}
\newtheorem{conv}[thm]{Convention}

\DeclareMathOperator{\ch}{\sf CH}
\DeclareMathOperator{\MA}{\sf MA}
\DeclareMathOperator{\zfc}{\sf ZFC}
\DeclareMathOperator{\gch}{\sf GCH}
\DeclareMathOperator{\pr}{Pr}
\DeclareMathOperator{\im}{Im}

\DeclareMathOperator{\otp}{otp}
\DeclareMathOperator{\dom}{dom}
\DeclareMathOperator{\acc}{acc}
\DeclareMathOperator{\non}{non}
\newcommand{\s}{\subseteq}
\newcommand{\br}{\blacktriangleright}
\newcommand{\forcingname}[1]{\dot #1}

\newcommand{\one}{\mathop{1\hskip-3pt {\rm l}}}
\newcommand{\forces}[2]{\Vdash_{#1} \mbox{``} #2 \mbox{''}}
\newcommand{\proo}{\prz{\aleph_1}{\aleph_0}{1}{\aleph_1}{\aleph_0}}
\newcommand{\notimplies}{\mathrel{{\ooalign{\hidewidth$\not\phantom{=}$\hidewidth\cr$\implies$}}}}
\newcommand{\prz}[5]{\pr_0(#1,\allowbreak\faktor{{\scriptstyle{#2\circledast#1}}}{ {}^{#3\circledast#1}},\allowbreak#4,#5)}
\newcommand{\pro}[5]{\pr_1(#1,\allowbreak\faktor{{\scriptstyle{#2\circledast#1}}}{ {}^{#3\circledast#1}},\allowbreak#4,#5)}
\newcommand{\proon}{\pr_1(\aleph_1,\faktor{{\scriptstyle{\aleph_0\circledast\aleph_1}}}{ {}^{1\circledast\aleph_1}},\aleph_1,\aleph_0)}
\renewcommand\mid{\mathrel{|}\allowbreak}

\subjclass[2010]{Primary 03E02; Secondary 03E35, 03E17}

\title[Strongly Luzin sets and partition relations]{Ramsey theory over partitions III:\\Strongly Luzin sets and partition relations}

\author[M. Kojman]{Menachem Kojman}
\address{Department of Mathematics, Ben-Gurion University of the Negev, P.O.B. 653, Be’er Sheva, 84105 Israel}
\urladdr{https://www.math.bgu.ac.il/~kojman/}

\author[A. Rinot]{Assaf Rinot}
\address{Department of Mathematics, Bar-Ilan University, Ramat-Gan 5290002, Israel.}
\urladdr{http://www.assafrinot.com}

\author[J. Stepr\={a}ns]{Juris Stepr\={a}ns}
\address{Department of Mathematics \& Statistics, York University, 4700 Keele Street, Toronto, Ontario, Canada M3J 1P3}
\urladdr{http://www.math.yorku.ca/~steprans/}

\date{Preprint as of April 16, 2022. For the latest version, visit \textsf{http://assafrinot.com/paper/55}.}

\begin{document}
\begin{abstract} 
The strongest  type of  coloring of pairs of countable ordinals, gotten by 
Todor{\v{c}}evi{\'c} from a strongly Luzin set,
is shown to be equivalent to the existence of a nonmeager set of reals of size $\aleph_1$.
In the other direction, it is shown that the existence of both a strongly Luzin set and a coherent Souslin tree 
is compatible with the existence of a countable partition of pairs of countable ordinals such that no coloring is strong over it.

This clarifies the interaction between a gallery of coloring assertions
going back to Luzin and Sierpi\'nski a hundred years ago.
\end{abstract}
\maketitle

\section{Introduction}

Each of the  following propositions is a consequence of Cantor's Continuum Hypothesis ($\ch$):
\begin{itemize}
\item[(M)] There is a nonmeager set of reals of size $\aleph_1$;
\item[(L)] There is an uncountable set of reals whose
  intersection with every meager set is countable;
\item[(L*)] There is an uncountable set of reals $X$ such that, for every positive integer $d$,
and every meager subset $Y$ of $\mathbb R^d$, the intersection $Y\cap X^d$
contains no uncountable pairwise disjoint subfamily;\footnote{Two $d$-tuples $(p_1,\ldots,p_d)$ and $(q_1,\ldots,q_d)$ are understood to be \emph{disjoint} iff $\{ p_1,\ldots,p_d\}\cap\{q_1,\ldots,q_d\}\neq\emptyset$.}
\item[(S)] There is a sequence $\langle f_n\mid n\in\mathbb N\rangle$ of
  functions from from $\aleph_1$ to $\aleph_1$ such that for every
  uncountable $I\s\aleph_1$, for all but finitely many $n$'s,
  $f_n[I]=\aleph_1$;
\item[(EHM)] There is a coloring $c:[\aleph_1]^2\rightarrow\aleph_1$
  such that for every infinite $A\subseteq \aleph_1$ and every
  uncountable $B\subseteq \aleph_1$ there is $\alpha\in A$ such
  that $c[\{\alpha\}\times B]=\aleph_1$;
\item[(G)] There is a coloring $c:[\aleph_1]^2\rightarrow2$
  such that for every uncountable pairwise disjoint family $\mathcal B\s[\aleph_1]^{<\aleph_0}$
  and every $\delta<2$, there are $a,b\in\mathcal B$ with $\max(a)<\min(b)$ such that $c[a\times b]=\{\delta\}$;
\item[(T)] There is a coloring $c:[\aleph_1]^2\rightarrow\aleph_1$
  such that for all $k,l<\omega$,  
  for every infinite pairwise disjoint family $\mathcal A\s[\aleph_1]^{k}$
  and every uncountable pairwise disjoint family $\mathcal B\subseteq[\aleph_1]^{l}$
  there is $a\in\mathcal A$ such that 
  for every function $f:k\times l\rightarrow\aleph_1$, there is
  $b\in\mathcal B$ such that $c(a(i),b(j))=f(i,j)$ for all $(i,j)\in k\times l$.
\end{itemize}

 (L) was derived from $\ch$ by
Mahlo and independently by Luzin around 1913;
such a set of reals is called a \emph{Luzin set}.
 (L*) was derived by Todor{\v{c}}evi{\'c} \cite[p.~51]{TodPartTop},
and such a set was named \emph{strongly Luzin}.
 (S) was derived by Sierpi\'nski in 1932, and may be found in his monograph \cite{sierpinski1934hypothese}.
 (EHM) was derived by Erd\H{o}s, Hajnal and Milner in 1966 \cite{EHM}
 and (G) was derived by Galvin in 1980 \cite{galvin}.
A special case of (T) in which either $k$ or $l$ is $1$ and the number of colors is just $2$ was gotten from $\ch$ by Hajnal and Juh\'asz \cite{MR336705} in their work on HFC and HFD spaces from the 1970's .

Evidently, (L*)$\implies$(L)$\implies$(M), (T)$\implies$(G) 
and (T)$\implies$(EHM)$\implies$(S). 
In 1980, Shelah \cite{Sh:100,JdSh:478} established that (M)$\notimplies$(L). 
In 1987, Todor{\v{c}}evi{\'c}  \cite[pp.~290--291]{TodActa} proved that (L)$\implies$(S)$\iff$(EHM).
In 1989, Todor{\v{c}}evi{\'c} \cite[Proposition~6.4]{TodPartTop} proved that (L*)$\implies$(T).
Recently, Miller \cite{Miller} proved (S)$\implies$(M) and 
 Guzm\'an \cite{MR3694336} proved (M)$\implies$(S), establishing (M)$\iff$(S).

The first main result of this paper expands this circle of equivalences:

\begin{thma} (M)$\iff$(T).
\end{thma}

Proposition (T), 
which in the language of Definition~\ref{prelations} below is
denoted $$\proo$$ asserts the existence of an extremely strong coloring,
yet one can ask for even more. 
The notion of a strong coloring \emph{over a partition} $p$  was introduced
recently in \cite{strongcoloringpaper},
where it was shown that for every strong coloring $c:[\aleph_1]^2\rightarrow\aleph_1$
there is a partition $p:[\aleph_1]^2\rightarrow 2$ 
such that $c$ is not strong over $p$. Nevertheless, by \cite[Lemma~9]{strongcoloringpaper}, 
if the space of strong colorings which witness $\proo$ is non-empty, then for every partition
$p:[\aleph_1]^2\to \mu$ with $(\aleph_1)^\mu=\aleph_1$ there is a strong
coloring which witnesses it over $p$; the existence of such a coloring is denoted by $$\proo_p .$$

Altogether, by Theorem~A, (M) implies $\proo_p$ for all  partitions $p:[\aleph_1]^2\rightarrow \theta$ with a finite $\theta$, 
and $\ch$ implies $\proo_p$ also for $\theta=\aleph_0$.
By Theorem~\ref{ThePr0Cohen} below, $\proo_p$ can hold for all $p:[\aleph_1]^2\rightarrow \aleph_0$ also in models with an arbitrarily large continuum.

It is natural to ask, then, whether (M)  implies $\proo_p$ for all
countable  $p$. 
We prove in Section~\ref{nonmeager} below:

\begin{thmap} (M) is equivalent to $\proo_p$  for all 
$\ell_\infty$-coherent partitions $p:[\aleph_1]^2\rightarrow\aleph_0$.
\end{thmap}

This leaves open the question whether 
Theorem~A' can be extended to cover all countable partitions. The second main result of this paper shows that this is not
the case. In fact, (L*) does not even imply $\aleph_1\to_p[\aleph_1]^2_{\aleph_0}$ for all countable  $p$, whereas without the $p$ this relation holds in $\zfc$
by Todor{\v{c}}evi{\'c}'s celebrated theorem \cite{TodActa}.

\begin{thmb}
It is consistent that (L*) holds and there is a partition
$p:[\aleph_1]^2\rightarrow\aleph_0$ such that the positive Ramsey relation $\aleph_1\rightarrow_p(\aleph_1)^2_{\aleph_0}$ holds.
\end{thmb}

Theorem~B is proved in Section~\ref{Luzin}.
The model witnessing the theorem is obtained by forcing over a ground model of $\ch$
to which the partition $p$, whose existence is equivalent to the statement  $\mathfrak d=\aleph_1$, belongs.
In the forcing extension there  exist a strongly Luzin set and  a coherent Souslin tree,
but every coloring $c:[\aleph_1]^2\rightarrow\aleph_0$ 
is \emph{$p$-special}:
 there is a decomposition $\langle X_i\mid i<\omega\rangle$ of $\aleph_1$ into $(p,c)$-homogeneous sets, that is, 
 for each $i<\omega$,  $p(\alpha,\beta)$ determines $c(\alpha,\beta)$ for  all $(\alpha,\beta)\in [X_i]^2$.

\section{Strong colorings and partitions}\label{sectionP}

Surveys of the rich theory
of strong colorings that was developed since Sierpi{\'n}ski's time to
the present may be found in the introductions to
\cite{paper18,strongcoloringpaper}. For the scope of this paper, 
we just need the following.
\begin{defn}[\cite{paper50}]\label{prelations}
Let $p:[\omega_1]^2\rightarrow\omega$ be a partition. 
For cardinals $\chi\le\omega$ and $\theta\le\omega_1$, a coloring $c:[\omega_1]^2\rightarrow\theta$ is said to witness
\begin{itemize}
\item $\pr_1(\omega_1,\omega_1,\theta,\chi)_p$ iff for every uncountable pairwise
  disjoint subfamily $\mathcal A\s[\omega_1]^{n}$, every $n<\chi$,  
and every function $\tau:\omega\to\theta$ there are $a,b\in\mathcal A$ with $a<b$ such that
$$c(\alpha,\beta)=\tau(p(\alpha,\beta))\text{ for all }\alpha\in a\text{ and }\beta\in b;$$
\item $\pro{\omega_1}{\omega}{1}{\theta}{\chi}_p$
iff for every pairwise disjoint subfamilies $\mathcal A,\mathcal B$ of $[\omega_1]^n$ 
with $|\mathcal A|=\omega$, $|\mathcal B|=\omega_1$ and $n<\chi$
there is $a\in \mathcal A$ such that
for every function $\tau:\omega\to\theta$,
there is $b\in\mathcal B$ with $a<b$ such that
$$c(\alpha,\beta)=\tau(p(\alpha,\beta))\text{ for all }\alpha\in a\text{ and }\beta\in b;$$

\item $\prz{\omega_1}{\omega}{1}{\theta}{\chi}_p$
iff for every pairwise disjoint subfamilies $\mathcal A,\mathcal B$ of $[\omega_1]^n$ 
with $|\mathcal A|=\omega$, $|\mathcal B|=\omega_1$ and $n<\chi$,
there is $a\in \mathcal A$ such that
for every matrix $(\tau_{i,j})_{ i,j<n}$ of
functions from $\omega$ to $\theta$,
there is $b\in\mathcal B$ with $a<b$ such that
$$c(a(i),b(j))=\tau_{i,j}(p(a(i),b(j)))\text{ for all }i,j<n.$$
\end{itemize}
\end{defn}
\begin{remark} Here $a(i)$ stands for the $i^{th}$ element of $a$.
\end{remark}

\begin{defn} For a partition $p:[\omega_1]^2\rightarrow\omega$:
\begin{itemize}
\item $p$ has \emph{injective fibers} iff $p(\alpha,\beta)\neq p(\alpha',\beta)$ for all $\alpha<\alpha'<\beta$;
\item $p$ has \emph{finite-to-one fibers} iff $\{ \alpha<\beta\mid p(\alpha,\beta)=j\}$ is finite for all $\beta<\omega_1$ and $j<\omega$;
\item $p$ has \emph{almost-disjoint fibers}  iff  $\{p(\alpha,\beta)\mid \alpha<\beta\}\cap\{ p(\alpha,\beta')\mid \alpha<\beta\}$ is finite
for all $\beta<\beta'<\omega_1$;
\item $p$ has \emph{Cohen fibers} 
iff for every injection $g:a\rightarrow\omega$ with $a\in[\omega_1]^{<\omega}$,
there are cofinally many $\beta<\omega_1$ such that $g(\alpha)=p(\alpha,\beta)$ for all $\alpha\in a$;
\item $p$ is \emph{coherent} iff  $\{\alpha<\beta\mid p(\alpha,\beta)\neq p(\alpha,\beta')\}$ 
is finite for all $\beta<\beta'<\omega_1$;
\item $p$ is \emph{$\ell_\infty$-coherent}
iff $\{p(\alpha,\beta)-p(\alpha,\beta')\mid \alpha<\beta\}$ is finite for all $\beta<\beta'<\omega_1$.
\end{itemize}
\end{defn}
\begin{remark}
An example of an $\ell_\infty$-coherent partition which is not coherent 
is the map $\rho_2:[\omega_1]^2\rightarrow\omega$ from the theory of \emph{walks on ordinals} \cite[p.~269]{TodActa}.
\end{remark}

For many cardinal characteristics  $\mathfrak x$ of the continuum, the
assertion ``$\mathfrak x=\aleph_1$'' may be reformulated as a
statement about the existence of a partition
$p:[\omega_1]^2\rightarrow\omega$ with certain properties.  In
Section~\ref{Luzin} we shall need following  reformulation of
``$\mathfrak d=\aleph_1$''.

\begin{lemma} \label{jsdbtgy}
$\mathfrak d=\aleph_1$ iff there exists a partition
  $p:[\omega_1]^2\rightarrow\omega$ with injective,
  almost-disjoint and Cohen fibers which satisfies the following:

For every function $h:\epsilon\rightarrow\omega$ with $\epsilon<\omega_1$
there exists $\gamma<\omega_1$
such that for every $b\in[\omega_1\setminus\gamma]^{<\aleph_0}$
  there exists $\Delta\in[\epsilon]^{<\aleph_0}$ such that:
  \begin{itemize}
  \item for all $\alpha\in\epsilon\setminus\Delta$  and $\beta\in b$, $h(\alpha)<p(\alpha,\beta)$;
  \item $p\restriction ((\epsilon\setminus\Delta)\times b)$ is injective.
  \end{itemize}
\end{lemma}
\begin{proof} For the backwards implication, derive an $\omega_1$-sized cofinal family
$\{ r_\beta\mid \omega\le\beta<\omega_1\}$ in
  $({}^\omega\omega,{<^*})$ by letting $r_\beta(n):=p(n,\beta)$.  

We turn now to  the forward implication.
Fix a coherent $q:[\omega_1]^2\rightarrow\omega$ having injective fibers (see, e.g., \cite[Theorem~5.9]{MR756630}). 
Fix an enumeration $\langle g_\beta\mid \beta<\omega_1\rangle$
of all injections $g$ with $\dom(g)\in [\omega_1]^{<\aleph_0}$ and $\im(g)\s \omega$ in which
each such injection occurs cofinally often.
For each $\beta<\omega_1$, let $m_\beta:=\sup(\im(g_\beta))+1$.
Fix a bijection $\pi:\omega\times\omega\leftrightarrow\omega$.
Derive a function $\psi:\omega\rightarrow\omega$ via 
$$\psi(m):=\max\{ i<\omega\mid \exists j<\omega\,(\pi(i,j)\le m)\}.$$

Using  $\mathfrak d=\aleph_1$, it is easy to construct recursively a sequence
$\vec d=\langle d_\beta\mid \beta<\omega_1\rangle$ such that
$\vec d$ is increasing and cofinal in $({}^\omega\omega,{<^*})$,
and, for every $\beta<\omega_1$, $\min(\im(d_\beta))>\psi(m_\beta)$.
Finally, define a partition $p:[\omega_1]^2\rightarrow\omega$ via:
$$p(\alpha,\beta):=\begin{cases}
g_\beta(\alpha)&\text{if }\alpha\in\dom(g_\beta);\\
\pi(d_\beta(q(\alpha,\beta)),q(\alpha,\beta))&\text{otherwise}.\end{cases}$$

\begin{claim} Let $\alpha<\beta<\omega_1$.
Then $\alpha\in\dom(g_\beta)$ iff $p(\alpha,\beta)\in\im(g_\beta)$.
\end{claim}
\begin{proof} The forward implication is clear, so suppose that $\alpha\notin\dom(g_\beta)$,
and set $j:=q(\alpha,\beta)$.
By the choice of $d_\beta$,  $i:=d_\beta(j)$ is greater than $\psi(m_\beta)$,
and hence $\pi(i,j)>m_\beta>\sup(\im(g_\beta))$. Altogether, $p(\alpha,\beta)=\pi(i,j)>\sup(\im(g_\beta))$.
\end{proof}
As $\pi$ is injective,  $q$ has injective fibers
and each $g_\beta$ is injective, 
it  follows that $p$ has injective fibers.
It is also clear that $p$ has Cohen fibers.

\begin{claim} $p$ has almost-disjoint fibers.
\end{claim}
\begin{proof} Fix an arbitrary pair $(\beta,\beta')\in[\omega_1]^2$ and consider the set
$$A:=\{p(\alpha,\beta)\mid \alpha<\beta\}\cap\{p(\alpha,\beta')\mid \alpha<\beta\}.$$

Evidently, $|A|\le m_\beta+m_{\beta'}+|\{ n<\omega\mid d_\beta(n)=d_{\beta'}(n)\}|<\omega$.
\end{proof}

To see that $p$ is as sought, fix arbitrary ordinal $\epsilon<\omega_1$ and 
function $h:\epsilon\rightarrow\omega$.
As $q$ has coherent fibers, for every $\beta<\omega_1$ above $\epsilon$,
the following set is finite
$$A^0_\beta:=\{\alpha<\epsilon\mid q(\alpha,\epsilon)\neq q(\alpha,\beta)\}.$$

Define a real $r:\omega\rightarrow\omega$ via 
$$r(n):=\begin{cases}
0&\text{if }\forall\alpha<\epsilon\,(q(\alpha,\epsilon)\neq n));\\
\psi(h(\alpha))&\text{if }q(\alpha,\epsilon)=n.
\end{cases}$$

Find a large enough ordinal $\gamma<\omega_1$ such that $\epsilon<\gamma$ and  $r<^* d_\beta$
for every $\beta\in[\gamma,\omega_1)$.
Now, let $b\in[\omega_1\setminus\gamma]^{<\aleph_0}$ be arbitrary.
As $q$ has injective fibers, for every $\beta\in[\gamma,\omega_1)$,
the following set is finite
$$A^1_\beta:=\{\alpha<\epsilon\mid r(q(\alpha,\beta))\ge d_\beta(q(\alpha,\beta))\}.$$

As $\vec d$ is $<^*$-increasing,
we may find some $m^*<\omega$ such that, for all $n\in[m^*,\omega)$ and $(\beta,\beta')\in[b]^2$, $d_\beta(n)<d_{\beta'}(n)$.
Now, as $q$ has injective fibers, it follows that 
the following set is finite:
$$\Delta:=\bigcup_{\beta\in b}(A^0_\beta\cup A^1_\beta\cup\dom(g_\beta)\cup\{\alpha<\epsilon\mid q(\alpha,\beta)<m^*\}).$$

\begin{claim} 
  \begin{enumerate}
  \item for all $\alpha\in\epsilon\setminus\Delta$  and $\beta\in b$, $h(\alpha)<p(\alpha,\beta)$;
  \item $p\restriction ((\epsilon\setminus\Delta)\times b)$ is injective.
  \end{enumerate}
\end{claim}
\begin{proof} 
(1) Let $\alpha\in\epsilon\setminus\Delta$ and $\beta\in b$.
Set $n:=q(\alpha,\epsilon)$. As $\alpha\in\epsilon\setminus A^0_\beta$,
$q(\alpha,\beta)=n$, 
so that $\psi(h(\alpha))=r(n)=r(q(\alpha,\beta))$.
As $\alpha\in\epsilon\setminus A^1_\beta$,
$r(q(\alpha,\beta))<d_\beta(q(\alpha,\beta))$.
Altogether, $\psi(h(\alpha))<d_\beta(q(\alpha,\beta))$
and hence $\pi(d_\beta(q(\alpha,\beta)),j)> h(\alpha)$ for all $j<\omega$.
In particular, since $\alpha\notin\dom(g_\beta)$, $p(\alpha,\beta)=\pi(d_\beta(q(\alpha,\beta)),q(\alpha,\beta))>h(\alpha)$.

(2) Fix $(\alpha,\beta),(\alpha',\beta')\in (\epsilon\setminus\Delta)\times b$ with $p(\alpha,\beta)=p(\alpha',\beta')$.
If $\beta=\beta'$, then since $p$ has injective injective fibers, $\alpha=\alpha'$ and we are done.
So, suppose that $\beta\neq\beta'$, say, $\beta<\beta'$. 
Denote $(k,n):=(d_\beta(q(\alpha,\beta)),q(\alpha,\beta))$.
As $p(\alpha,\beta)=p(\alpha',\beta')$, $\alpha\notin\dom(g_\beta)$ and $\alpha'\notin\dom(g_{\beta'})$,
it follows that $(d_{\beta'}(q(\alpha',\beta')),q(\alpha',\beta'))=(k,n)$.
In particular, $d_\beta(n)=d_{\beta'}(n)$.
As $\alpha\in\epsilon\setminus\Delta$, we infer that $n \ge m^*$, so  $d_\beta(n)<d_{\beta'}(n)$. This is a contradiction.
\end{proof}
This completes the proof.
\end{proof}

\section{Strong colorings  from a nonmeager set}\label{nonmeager}

In the next Theorem, which proves Theorems A and A':
Clause~(1) is proposition~(M).
Clause~(2) is a syntactic weakening of proposition~(S), but
addressing a concern raised by Bagemihl and Sprinkle \cite{MR63420},
it was shown by Miller \cite{Miller} to be equivalent to it.
Clause~(3) is a high-dimensional version of Clause~(2).
Clause~(4) asserts the existence of a coloring $c:[\omega_1]^2\rightarrow\omega_1\times\omega$ 
for which the map $(\alpha,\beta)\mapsto\delta$ iff $\exists\iota[c(\alpha,\beta)=(\delta,\iota)]$ 
witnesses $\proon$ of Definition~\ref{prelations} and $c$ itself has finite-to-one fibers.
Clause~(5) is proposition~(T) over $\ell_\infty$-coherent partitions.
Clause~(6) is slightly weaker than proposition~(EHM).
The implication $(7)\implies(1)$ is due to Miller \cite{MR613787}
and the implication $(1)\implies(2)$ is due to Guzm\'an \cite{MR3694336}.

\begin{thm}\label{thma} All of the following are equivalent:
\begin{enumerate}
\item $\non(\mathcal  M)=\aleph_1$;
\item There exists a sequence $\vec f=\langle f_m\mid m<\omega\rangle$ of functions
from $\omega_1$ to $\omega_1$ satisfying
that for every cofinal subset $B\s\omega_1$ there exists $m<\omega$ such that
$f_m[B]=\omega_1$;
\item There exists a sequence $\vec g=\langle g_n\mid n<\omega\rangle$ of functions
from $\omega_1$ to $\omega_1$ satisfying
that for every uncountable pairwise disjoint subfamily $\mathcal B\s[\omega_1]^{<\aleph_0}$
there are infinitely many $n<\omega$ such that for every
$\gamma<\omega_1$, for some $b\in\mathcal B$, $g_n[b]=\{\gamma\}$;
\item
 There exists a coloring $c:[\omega_1]^2\to\omega_1\times\omega$ with finite-to-one fibers,
such that for every 
\begin{itemize}
 \item $k<\omega$ and an infinite pairwise disjoint subfamily  $\mathcal A\s[\omega_1]^k$
 \item $l<\omega$ and an uncountable pairwise disjoint subfamily $\mathcal B\s[\omega_1]^l$
 \end{itemize}
 there exists
 $a\in\mathcal A$ such that for every $\delta<\omega_1$
 there are $\iota<\omega$ and $b\in\mathcal B$ such that
 $$\{\alpha<\beta\mid c(\alpha,\beta)=(\delta,\iota)\}=a\text{ for every }\beta\in b;$$
\item For every $\ell_\infty$-coherent partition $p:[\omega_1]^2\rightarrow\omega$,
there exists a corresponding coloring $d:[\omega_1]^2\to\omega_1$
satisfying that for every
\begin{itemize}
 \item $k<\omega$ and an infinite pairwise disjoint subfamily  $\mathcal A\s[\omega_1]^k$
 \item $l<\omega$ and an uncountable pairwise disjoint subfamily $\mathcal B\s[\omega_1]^l$
 \end{itemize}
 there exists $a\in\mathcal A$ such that
 for every matrix $\langle \tau_{n,m}\mid n<k,m<l\rangle$ of
functions from $\omega$ to $\omega_1$
there exists $b\in\mathcal B$ such that 
$$d(a(n),b(m))=\tau_{n,m}(p(a(n),b(m)))\text{ for all }n<k\text{ and }m<l;$$
\item
  There exists a coloring $e:[\omega_1]^2\rightarrow \omega$ such that for every infinite
 $A\s\omega_1$ and uncountable $B\s\omega_1$, there is $\alpha\in A$ such that $\{e(\alpha,\beta)\mid \beta\in B\setminus(\alpha+1)\}=\omega$;
\item
  There exists a subset $X\s{}^\omega\omega$ of size $\aleph_1$ with the property that
for every real $y:\omega\rightarrow\omega$,
for some $x\in X$, $x\cap y$ is infinite.
\end{enumerate}
\end{thm}

\begin{proof}  For the rest of the proof we fix a bijection $\pi:\omega\leftrightarrow \omega\times \omega$
and, by the Engelking-Karlowicz theorem \cite{ek}, we fix  a sequence $\langle h_j\mid j<\omega\rangle$
of functions from $\omega_1$ to $\omega$
such that  for every set $x\in[\omega_1]^{<\aleph_0}$ and a function
$h:x\rightarrow\omega$ there exists $j<\omega$ such that $h\s h_j$.

$(1)\implies(2)$: This is Proposition~2.2 of \cite{MR3694336}.

$(2)\implies(3)$:
Let $\vec f$ witness  Clause~(2).  For every $\beta<\omega_1$ fix a surjection $e_\beta:\omega\rightarrow\beta+1$.
Define a sequence $\vec g=\langle g_n\mid n<\omega\rangle$ of functions from
$\omega_1$ to $\omega_1$, as follows.
Given $n<\omega$, let $(m,j):=\pi(n)$ and  for every $\beta<\omega_1$ set
$$g_n(\beta):=f_m(e_\beta(h_j(\beta))).$$

To see that $\vec g$ witnesses Clause~(3), fix an arbitrary
uncountable pairwise disjoint subfamily $\mathcal
B\s[\omega_1]^{<\aleph_0}$ and some $k<\omega$. We shall find an
integer $n>k$ such that, for every $\gamma<\omega_1$ there is some
$b\in\mathcal B$ such that $g_n[b]=\{\gamma\}$.

For every $b\in\mathcal B$, define a function $h^b:b\rightarrow\omega$ via:
$$h^b(\beta):=\min\{i<\omega\mid e_\beta(i)=\min(b)\}.$$
Fix $j'<\omega$ for which $\mathcal B':=\{ b\in\mathcal B\mid h^b\s h_{j'}\}$
is uncountable, and then let $$m':=\max(\{0\}\cup\{m<\omega\mid \pi^{-1}(m,j')\le k\}).$$
Evidently, $B_0:=\{\min(b)\mid b\in\mathcal B'\}$ is uncountable.
Next, for every $i\le m'$ such that $B_i$ has already been defined, proceed as 
follows:

$\br$  If $Z_i:=\{ \beta\in B_i\mid f_i(\beta)\neq 0\}$
is uncountable, then let $B_{i+1}:=Z_i$;

$\br$ Otherwise, let $B_{i+1}:=B_i\setminus Z_i$.

In either case, $B_{i+1}\s B_i$ is uncountable with
$f_i[B_{i+1}]\neq\omega_1$.

Finally, as $B_{m'+1}$ is uncountable,  let us pick, by the choice of $\vec f$,
an integer $m<\omega$ such that $f_m[B_{m'+1}]=\omega_1$.

For all $i\le m'$, $f_i[B_{m'+1}]\s f_i[B_{i+1}]\subsetneq\omega_1$, so
$m>m'$.
In particular, $n:=\pi(m,j')$ is larger than $k$.
To see that $n$ is as sought, let $\gamma\in \omega_1=f_m[B_{m'+1}]$ be
arbitrary. Pick $\beta'\in B_{m'+1}$ with $f_m(\beta')=\gamma$.
As $\beta'\in B_{m'+1}\s B_0$, let us pick $b\in\mathcal B$ such that
$h^b\s h_{j'}$ and $\min(b)=\beta'$.
Let $\beta\in b$ be arbitrary. Then
$$g_{n}(\beta)=f_m(e_\beta(h_{j'}(\beta)))=f_m(e_\beta(h^b(\beta)))=f_m(\beta') =\gamma.$$
So $g_{n}[b]=\{\gamma\}$, as required.

$(3)\implies(4)$: The proof here is inspired by Miller's proof of \cite[Proposition~4]{Miller}.
Define an eventually increasing sequence of integers $\langle m_n\mid
n<\omega\rangle$ by recursion, setting $m_0:=1$, and $m_{n+1}:=n!\cdot
(\sum_{i\le n}m_i)$ for every $n<\omega$.  For every $n<\omega$, let
$\Phi_n:=\bigcup\{{}^x(\omega_1\times\omega)\mid x\s \omega_1,
|x|=m_n\}$.  Evidently, $|\Phi_n|=\omega_1$, so we may fix an
injective enumeration $\langle \phi_n^\gamma\mid
\gamma<\omega_1\rangle$ of $\Phi_n$.

Let $\vec g$ witness Clause~(3).
Define a coloring $d:[\omega_1]^2\rightarrow(\omega_1\times\omega)\times\omega$ by letting for all
$\alpha<\beta<\omega_1$:
$$d(\alpha,\beta):=\begin{cases}
((\alpha,0),0)&\text{if }\alpha\notin\bigcup_{i<\omega}\dom(\phi_i^{g_i(\beta)});\\
(\phi_n^{g_n(\beta)}(\alpha),n+1)&\text{if }n=\min\{i<\omega\mid
\alpha\in\dom(\phi_i^{g_i(\beta)})\}.
\end{cases}$$

Finally, define a coloring $c:[\omega_1]^2\rightarrow\omega_1\times\omega$ by letting $c(\alpha,\beta):=(\gamma,\pi(\iota,n))$
iff $d(\alpha,\beta)=((\gamma,\iota),n)$. It is clear that $d$ has finite-to-one fibers,
and hence so does $c$.
To see that $c$ witnesses Clause~(4), fix positive integers $k,l$ along with $\mathcal A,\mathcal B$ such that:
\begin{itemize}
 \item $\mathcal A$ is an infinite pairwise disjoint subfamily  of $[\omega_1]^k$,
 \item $\mathcal B$ is an uncountable pairwise disjoint  subfamily of $[\omega_1]^l$, and
\item $a<b$ for all $a\in\mathcal A$ and $b\in\mathcal B$.
 \end{itemize}

By the choice of $\vec g$, let us fix an integer $n>\max\{k,l\}$
such that, for every $\gamma<\omega_1$, for some $b\in\mathcal B$,
$g_{n+1}[b]=\{\gamma\}$.
As $m_{n+1}$ is divisible by $k$,
we now fix an injective sequence $\langle a_\iota\mid \iota<{m_{n+1}\over k}\rangle$ consisting of elements of $\mathcal A$.

\begin{claim} There exists $\iota<{m_{n+1}\over k}$ such that, for every
$\delta<\omega_1$, there is $b\in\mathcal B$,
such that, for every $\beta\in b$:
 $$\{\alpha<\beta\mid c(\alpha,\beta)=(\delta,\pi(\iota,n+2))\}=a_\iota.$$
\end{claim}
\begin{proof} Suppose not. Then, for every $\iota<{m_{n+1}\over k}$,
we may find some $\delta_\iota<\omega_1$ such that for all $b\in\mathcal B$, for some $\beta\in b$,
 $$\{\alpha<\beta\mid d(\alpha,\beta)=((\delta_\iota,\iota),n+2)\}\neq a_\iota.$$
Define a function $\phi:\biguplus\{a_\iota\mid \iota<{m_{n+1}\over k}\}\rightarrow\omega_1\times\omega$ by letting
$\phi(\alpha):=(\delta_\iota,\iota)$ iff $\alpha\in a_\iota$.
As $|\biguplus\{a_\iota\mid \iota<{m_{n+1}\over k}\}|={m_{n+1}\over k}\cdot k=m_{n+1}$, we infer that
$\phi\in\Phi_{n+1}$,
so we may fix $\gamma<\omega_1$ such that $\phi=\phi_{n+1}^\gamma$. Now, pick
$b\in\mathcal B$ with $g_{n+1}[b]=\{\gamma\}$.

For every $i\le n$ and $\beta\in b$, let
$x_i^\beta:=\dom(\phi_i^{g_i(\beta)})$,
so that $|x_i^\beta|=m_i$.
Next, set $x:=\bigcup\{ x_i^\beta\mid i\le n,\beta\in b\}$.
As $|b|=l$, we infer that $|x|\le l\cdot \sum_{i\le n}m_i$.
Thus $$k\cdot |x|\le  k\cdot l\cdot \sum_{i\le n}m_i< n!\cdot \sum_{i\le n}m_i=m_{n+1}.$$
In particular, $|x|<{m_{n+1}\over k}$, so we may fix $\iota<{m_{n+1}\over k}$ such that $a_\iota\cap x=\emptyset$.

Let $\beta\in b$ be arbitrary. Consider the set
$$A:=\{\alpha<\beta\mid d(\alpha,\beta)=((\delta_\iota,\iota),n+2)\}.$$
As $g_{n+1}(\beta)=\gamma$, we infer that $\phi_{n+1}^{g_{n+1}(\beta)}=\phi$, so, by the definition of $d$:
 $$A\s \{\alpha<\beta\mid \phi_{n+1}^{g_{n+1}(\beta)}(\alpha)=(\delta_\iota,\iota)\}\s \{\alpha<\beta\mid \phi(\beta)(\alpha)=(\delta_\iota,\iota)\}=a_\iota.$$

On the other hand, for every $\alpha\in a_\iota\s \dom(\phi_{n+1}^{g_{n+1}(\beta)})$,
as $\alpha\notin x$, it follows that $\min\{i<\omega\mid \alpha\in\dom(\phi_i^{g_i(\beta)})\}=n+1$,
and hence
$$d(\alpha,\beta)=(\phi_{n+1}^{g_{n+1}(\beta)}(\alpha),n+2)=(\phi(\alpha),n+2)=((\delta_\iota,\iota),n+2),$$
so that $\alpha\in A_i$.
Altogether, $A=a_\iota$, contradicting the choice of $\delta_\iota$.
\end{proof}

$(4)\implies(5)$: Fix $c$ witnessing Clause~(4).  
Let $\langle \eta_\gamma\mid \gamma<\omega_1\rangle$ be some injective enumeration
of $\bigcup\{{}^{k\times l\times t}\omega_1\mid k,l,t<\omega\}$ and
let $\langle (i_\delta,j_\delta,\gamma_\delta)\mid \delta<\omega_1\rangle$ be
some injective enumeration of
$\omega\times\omega\times\omega_1$,

Now, given any $\ell_\infty$-coherent partition $p:[\omega_1]^2\rightarrow\omega$,
define a coloring $d:[\omega_1]^2\rightarrow\omega_1$ as follows.
Given $(\alpha,\beta)\in[\omega_1]^2$, let $(\delta,\iota):=c(\alpha,\beta)$
and then set $$d(\alpha,\beta):=\begin{cases}\eta_{\gamma_\delta}(h_{i_\delta}(\alpha),h_{j_\delta}(\beta),p(\alpha,\beta))&\text{if }(h_{i_\delta}(\alpha),h_{j_\delta}(\beta),p(\alpha,\beta))\in\dom(\eta_{\gamma_\delta})\\
0&\text{otherwise}.\end{cases}$$

To see that $d$ witnesses Clause~(5), fix $k,l,\mathcal A,\mathcal B$ and $\epsilon<\omega_1$ such that:
\begin{itemize}
 \item $\mathcal A$ is an infinite pairwise disjoint subfamily  of $[\omega_1]^k$,
 \item $\mathcal B$ is an uncountable pairwise disjoint subfamily of $[\omega_1]^l$ and
\item $\max(a)<\epsilon\le\min(b)$ for all $a\in\mathcal A$ and all $b\in\mathcal B$.
 \end{itemize}

For
every $x\in\mathcal A\cup\mathcal B$, define a function
$h^x:x\rightarrow\omega$ via:
$$h^x(\beta):=\otp(x\cap\beta).$$
Now pick $j'<\omega$ for which
$\mathcal B':=\{ b\in\mathcal B\mid h^b\s h_{j'}\}$ is uncountable.
As $p$ is $\ell_\infty$-coherent, we may shrink $\mathcal B'$ further and assume the existence of some $q<\omega$ such that, for all $b\in\mathcal B'$:
$$\{|p(\alpha,\epsilon)-p(\alpha,\beta)|\mid \beta\in b\}\s q.$$

Now, as $|\mathcal A|=\aleph_0$ and $|\mathcal B'|=\aleph_1$, by the choice
of $c$, we may fix $a\in\mathcal A$ such that, for every
$\delta<\omega_1$, there are $b\in\mathcal B'$ and $\iota<\omega$ such that $c[a\times b]=\{(\delta,\iota)\}$.

\begin{claim}
Let $\langle \tau_{n,m}\mid n<k,m<l\rangle$ be a matrix of functions from $\omega$ to $\omega_1$.
Then there exists $b\in\mathcal B'$ satisfying that,  for all $n<k$ and $m<l$,
$$d(a(n),b(m))=\tau_{n,m}(p(a(n),b(m))).$$
\end{claim}
\begin{proof} Fix $i'<\omega$ such that $h^a\s h_{i'}$.
Let $t:=\max\{p(\alpha,\epsilon)+q\mid \alpha\in a\}$.
Define a function $\eta:k\times l\times t\rightarrow\omega_1$ via:
$$\eta(n,m,s):=\tau_{n,m}(s).$$
Let $\delta<\omega_1$ be such that $(i_\delta,j_\delta,\eta_{\gamma_\delta})=(i',j',\eta)$.
Pick $b\in\mathcal B'$ and $\iota<\omega$ such that $c[a\times b]=\{(\delta,\iota)\}$.
Now, given $n<k$ and $m<l$, we have $c(a(n),b(m))=(\delta,\iota)$,

$p(a(n),b(m))< p(a(n),\epsilon)+q\le t$, so that
$$\begin{array}{lll}
d(a(n),b(m))&=&\eta_{\gamma_\delta}(h_{i_\delta}(a(n)),h_{j_\delta}(b(m)),p(a(n),b(m)))\\
&=&\eta(h^a(a(n)),h^b(b(m)),p(a(n),b(m)))\\
&=&\eta(n,m,p(a(n),b(m)))\\
&=&\tau_{n,m}(p(a(n),b(m))),
\end{array}$$
as sought.
\end{proof}
$(5)\implies(6)$: Let $d$ witness Clause~(5) with respect to the constant partition $p:[\omega_1]^2\rightarrow1$.
Define a function $e:[\omega_1]^2\rightarrow\omega$ by letting
$e(\alpha,\beta):=d(\alpha,\beta)$ whenever $d(\alpha,\beta)<\omega$,
and $e(\alpha,\beta):=0$, otherwise.
Clearly, $e$ witnesses (6).

$(6)\implies(7)$: Let $e$ witness Clause~(6). Define $X=\{ x_\beta\mid
\beta<\omega_1\}$, as follows.  For every $\beta<\omega_1$, define a
function $x_\beta:\omega\rightarrow\omega$ via
$x_\beta(n):=e(n,\beta)$.  Towards a contradiction, suppose that
$y:\omega\rightarrow\omega$ is a counterexample.  It follows that
there exists a large enough $n<\omega$ for which
$B:=\{\beta<\omega_1\mid \dom(x_\beta\cap y)\s n\}$ is uncountable.
By the choice of $e$, we may now fix an integer $\alpha>n$ such that
$\{e(\alpha,\beta)\mid \beta\in B\setminus(\alpha+1)\}=\omega$.  In
particular, we may find $\beta\in B$ such that
$e(\alpha,\beta)=y(\alpha)$.  Altogether,
$x_\beta(\alpha)=e(\alpha,\beta)=y(\alpha)$ contradicting the fact
that $\beta\in B$ and $\alpha>n$.

$(7)\implies(1)$: By Theorem~1.3 of \cite{MR613787}.
\end{proof}

\begin{cor}[Theorem~A']\label{cor32} $non(\mathcal M)=\aleph_1$ iff $\proo_p$ holds for all $\ell_\infty$-coherent partitions $p:[\omega_1]^2\rightarrow\omega$.\qed
\end{cor}

\begin{cor} In the following, $(1)\implies(2)\implies(3)$ and none of the implications is revertible.
\begin{enumerate}
\item $\prz{\aleph_1}{\aleph_0}{1}{\aleph_1}{\aleph_0}$;
\item $\pr_0(\aleph_1,\aleph_1,\aleph_1,\aleph_0)$;
\item $\pr_0(\aleph_1,\aleph_1,\aleph_1,n)$ for all $n<\omega$.
\end{enumerate}
\end{cor}
\begin{proof} To see that (2) does not imply (1), recall that $\non(\mathcal M)>\aleph_1=\mathfrak b$ is consistent (e.g., after adding $\aleph_2$ random reals to a model of $\ch$)
and that Todor{\v{c}}evi{\'c} \cite{TodLC86} proved that Clause~(2) is a consequence of $\mathfrak b=\aleph_1$.

To see that (3) does not imply (2) recall that Clause~(2) is refuted by $\MA_{\aleph_1}$,
and that Peng and Wu \cite{MR3742590} proved Clause~(3) in $\zfc$.
\end{proof}

We conclude this section by pointing out that by a proof similar to that of \cite[Theorem~27]{strongcoloringpaper},
$\proo_p$ for all countable $p$ is compatible with the failure of $\ch$:

\begin{thm}\label{ThePr0Cohen}
In the forcing extension after adding adding $\aleph_2$ many Cohen reals, for every partition $p:[\omega_1]^2\rightarrow\omega$,
${\proo_{p}}$ holds. \qed
\end{thm}

\section{Strongly Luzin sets and strong colorings over partitions}\label{Luzin}
Luzin sets are tightly connected with strong colorings. In addition to Todor{\v{c}}evi{\'c}'s theorems
\cite[pp.~291]{TodActa},\cite[Proposition~6.4]{TodPartTop} that were improved by the main result of the previous section, an earlier result connecting Luzin sets with strong colorings 
may be found in \cite[Theorem~5.3]{MR500166}.
Likewise, Souslin trees give rise to strong colorings (see \cite[Lemma~6.6]{jensen}, \cite[Lemma~1]{MR371662}, \cite[\S5]{TodPartTop}, and \cite[\S3]{MR3805665}),
and \emph{coherent} Souslin trees have further strong coloring applications (see \cite[\S3.3]{paper36}).
By \cite[\S6]{AbSh:403}, the existence of a Souslin tree does not imply the existence of a coherent one.

Now we show that a strongly Luzin set together with  a coherent Souslin tree do not suffice to entail $\aleph_1\nrightarrow_p[\aleph_1]^2_{\aleph_0}$ for all countable partitions $p$. 

\begin{thm}\label{nonmkjbs64hb}
It is consistent that all of the following hold simultaneously:
\begin{itemize}
\item There exists a strongly Luzin set;
\item There exists a coherent Souslin tree;
\item There exists a partition $p:[\omega_1]^2\rightarrow\omega$ such that all colorings $c:[\omega_1]^2\rightarrow\omega$ are $p$-special,
that is, 
for  every coloring $c:[\omega_1]^2\rightarrow\omega$, there is a
  decomposition $\omega_1=\biguplus_{i<\omega}X_i$ such that for all
  $i,j<\omega$, $c$ is constant over $\{ (\alpha,\beta)\in[ X_i]^2\mid p(\alpha,\beta)=j\}$.
\end{itemize}
\end{thm}

The model of Theorem~\ref{nonmkjbs64hb} will be the outcome of a finite support iteration of  posets
$\mathbb Q(p,c)$ of the following form.

\begin{defn}\label{def52} $\mathbb Q(p,c)$ consists of all triples $q=(a_q,f_q,w_q)$ satisfying all of the following:
\begin{enumerate}
\item $a_q\in[\omega_1]^{<\aleph_0}$;
\item $f_q:a_q\rightarrow\omega$ is a function;
\item $w_q$ is a function from a finite subset of $\omega\times\omega$ to $\omega$;
\item for all $(\alpha,\beta)\in[a_q]^2$, if $f_q(\alpha)=f_q(\beta)$, then  
$(f_q(\alpha),p(\alpha,\beta))\in\dom(w_q)$ and $c(\alpha,\beta)=w_q(f_q(\alpha),p(\alpha,\beta))$.
\end{enumerate}
\end{defn}

For a generic $G\subseteq \mathbb Q(p,c)$, let $X_{i,G} = \{\alpha<\omega_1\mid \exists q\in G\,(f_q(\alpha)=i)\}$.
It is not hard to see that for every partition $p:[\omega_1]^2\rightarrow\omega$ with injective and almost-disjoint fibers,
$\mathbb Q(p,c)$ has Property~$K$,\footnote{In fact, by a result from \cite[\S3]{paper49}, $\mathbb Q(p,c)$ satisfies the stationary-cc.} and 
for all $i,j<\omega$,
$$\one\forces{\mathbb Q(p,c)}{|\{c(\alpha,\beta)\mid (\alpha,\beta)\in [X_{i,\dot{G}}]^2 \text{ and } p(\alpha,\beta) = j\}| \le1} .$$

\begin{defn}\label{Newkjd67}
For all $q\in \mathbb Q(p,c)$, $k<\omega$ and  $z\in[\omega_1]^{<\aleph_0}$, define $q^\wedge (k, z)$ to be the triple $(a,f,w)$ satisfying:
\begin{itemize}
\item $a:=a_q\cup z$;
\item $f:a\rightarrow\omega$ is a function extending $f_q$
and satisfying $f(\alpha)=k+\otp(z\cap\alpha)$ for all $\alpha\in a\setminus a_q$;
\item $w_q:=w$.
\end{itemize}
\end{defn}

Note that $q^\wedge (k, z)$ may not be in $\mathbb Q(p,c)$,
but it will be, provided that $k\supseteq\im(f_q)$.

\begin{cor}\label{Akjh} For every $\beta<\omega_1$, $D_\beta:=\{ q\in\mathbb Q(p,c)\mid \beta\in a_q\}$ is dense,
so that  $$\one\forces{\mathbb Q(p,c)}{\biguplus_{i<\omega} X_{i,\dot{G}}=\omega_1} .$$
\end{cor}
\begin{proof}
Given arbitrary $q\in\mathbb Q(p,c)$  and $\beta<\omega_1$,
for all sufficiently large $k$,
$q^\wedge (k, \{\beta\})$ is a condition  in $D_\beta$, extending $q$.
\end{proof}

\begin{cor} $\one\forces{\mathbb Q(p,c)}{c\text{ is }p\text{-special}}$.\qed
\end{cor}

\begin{defn}\label{Pnvafsh} Let $p:[\omega_1]^2\to \omega$ be a partition. 
For any ordinal $\eta$, a finite-support iteration $\{\mathbb Q_\xi\}_{\xi\in \eta}$
will be called a \emph{$p$-iteration} iff $\mathbb Q_0$ is the trivial forcing,
and, for each ordinal $\xi$ with $\xi+1<\eta$ there is a
$\mathbb Q_\xi$-name  $\forcingname{{c}_\xi}$
 such that
 \begin{enumerate}
 \item $\one\forces{ \mathbb Q_\xi}{\forcingname{{c}_\xi}: [\omega_1]^2\to \omega\text{ is a coloring}}$,
 \item $ \mathbb Q_{\xi+1} = \mathbb Q_\xi * \mathbb Q(p,\forcingname{{c}_\xi})$.
 \end{enumerate}
\end{defn}

\begin{conv} If $\{\mathbb Q_\xi\}_{\xi\in \eta}$ is a $p$-iteration,
with $\eta>0$ a limit ordinal, then we denote its direct limit by $\mathbb Q_\eta$.
\end{conv}

From now on, we fix a $p$-iteration $\{{{\mathbb Q}}_\xi\}_{\xi\in \eta}$
for some partition $p:[\omega_1]^2\rightarrow\omega$ with injective and almost-disjoint fibers,
hence each of the iterands has Property~$K$,
and so does the whole iteration.

\begin{defn} A structure $\mathfrak M$ is said to be \emph{good} for the $p$-iteration $\{\mathbb Q_\xi\}_{\xi\in \eta}$
iff there is a large enough regular  cardinal $\kappa>\eta$ such that
all of the following hold:
\begin{itemize}
\item $\mathfrak M$ is a countable elementary submodel of $(\mathcal H_\kappa,{\in},\lhd_\kappa)$, where $\lhd_\kappa$ is a well-ordering of $\mathcal H_\kappa$;
\item $p,\{{{\mathbb Q}}_\xi\}_{\xi\in \eta}$ and $\{\forcingname{{c}}_\xi\mid \xi+1<\eta\}$ are in $\mathfrak M$.
\end{itemize}
\end{defn}

\begin{defn}\label{determined}
Define   $q\in \mathbb Q_\xi$ to be \emph{determined}  by recursion on $\xi\in\eta$:

$\br$ For $\xi = 0$, all the conditions are  determined. 

$\br$ For any $\xi$, a condition $q\in \mathbb Q_{\xi + 1}$ is  determined if:
\begin{enumerate}
\item $q\restriction \xi$ is  determined;
\item $q\restriction \xi\forces{ \mathbb Q_\xi}{q(\xi) =  (a_{q,\xi},f_{q,\xi},w_{q,\xi})}$ 
for an actual triple of finite sets;
\item for all $(\alpha, \beta)\in [a_{q,\xi}]^2$ there is $n<\omega$ such that $q\restriction \xi\forces{ \mathbb Q_\xi}{\dot{c}_\xi(\alpha,\beta) = {n}}$.
\end{enumerate}

$\br$ For any $\xi\in\acc(\eta)$, $q\in\mathbb Q_\xi$ is determined if $q\restriction\zeta$ is determined for all $\zeta<\xi$.
\end{defn}

By a standard argument, the  determined conditions are dense in $\mathbb Q_\eta$.

\begin{defn} \label{Catiyzblall8}
For a determined condition $q$ in the $p$-iteration, we say that $k$ is \emph{sufficiently large for $q$} iff
$k\supseteq\im( f_{q,\xi})$ for all $\xi$ in the support of $q$. 
\end{defn}
\begin{defn}
For a condition $q$ in the $p$-iteration,
$k<\omega$ and  $z\in[\omega_1]^{<\aleph_0}$,
define $q^\wedge (k,z)$ by letting $q^\wedge (k, z)(\xi) := q(\xi)^\wedge(k, z)$ for each $\xi$ in the support of $q$. 
\end{defn}

Note that if $q$ is determined and $k$ is sufficiently large for $q$, then for each $\xi$ in the support of $q$,
$q\restriction \xi\forces{\mathbb Q_\xi}{q^\wedge (k,z)\in {\mathbb Q}(p,\forcingname{{c}_\xi})}$. In effect, if $k$ is sufficiently large for $q$, then $q^\wedge(k, z)$ is a legitimate condition.

\begin{defn}\label{DejgkjsacRhjM}
For any structure $\mathfrak M$ good for the $p$-iteration $\{\mathbb Q_\xi\}_{\xi\in \eta}$,
for all $\xi\in\eta$ and a determined condition $q\in \mathbb Q_{\xi}$,
we define $q^\mathfrak{M}$, as follows.
The definition is by recursion on $\xi\in\eta$:

$\br$ For $\xi=0$ there is nothing to do.

$\br$ For any $\xi$ such that $q^\mathfrak{M}$ has been defined for all determined $q$ in $\mathbb Q_\xi$,
given a determined condition $q\in \mathbb Q_{\xi +1}$, we consider two cases:

$\br\br$ If $\xi\in\mathfrak M$, then let  $q^\mathfrak{M} := (q\restriction \xi)^\mathfrak{M} * (a_{q,\xi}\cap \mathfrak M,f_{q,\xi}\cap \mathfrak M,w_{q,\xi})$;

$\br\br$ Otherwise, just let $q^\mathfrak{M} := (q\restriction \xi)^\mathfrak{M}*(\emptyset,\emptyset,\emptyset)$.

$\br$ For any $\xi\in\acc(\eta)$, since this is a finite-support iteration, there is nothing new to define.
\end{defn}

If $q$ is determined, then, for every coordinate $\xi$ in the support of $q$,
$q^\mathfrak{M}(\xi)$ is a triple consisting of finite sets lying in $\mathfrak M$.
It is important to note that  $q^\mathfrak{M}$ may not, in general,  be a condition in $\mathbb Q_{\xi }$,
because the last clause of Definition~\ref{def52} may fail. 
Nevertheless,  $(q^\mathfrak{M})^\wedge (k,z) $ is a well-defined object, since its definition does not depend on  the $\dot{c}_\xi$'s. 

\begin{notation} 
For any determined condition $q\in \mathbb Q_\xi$, we denote by $A_q$ the union of $a_{q,\xi}$ over all $\xi$ in the support of $q$. 
\end{notation}

We now arrive at the main technical lemma of this section.

\begin{lemma}\label{Oawewlkn11}
Suppose  $p:[\omega_1]^2\to \omega$ satisfies the conclusion of Lemma~\ref{jsdbtgy}, and 
$\mathfrak M$ is a structure which is good for the $p$-iteration $\{\mathbb Q_\xi\}_{\xi\in \eta}$.

For all $\zeta\le\sup(\eta)$ and a determined condition $r\in \mathbb Q_\zeta$, 
there is a finite set $\bar{z}\subseteq \mathfrak M\cap\omega_1$ such that:
\begin{itemize}
\item[{\bf A:}] For every $z\in[\mathfrak M\cap\omega_1]^{<\aleph_0}$ covering $\bar z$ and every integer $k$ that is sufficiently large for $r$,
$(r^\mathfrak{M})^\wedge(k, z)$ is in $\mathfrak M\cap \mathbb Q_\zeta$ and is determined;
\item[{\bf B:}] For every $z\in[\mathfrak M\cap\omega_1]^{<\aleph_0}$ covering $ \bar{z}$
and every integer $k$ that is sufficiently large for $r$,
for the condition
$\bar r := (r^\mathfrak{M})^\wedge (k, z)$ and a condition $q\in \mathfrak M\cap\mathbb Q_\zeta$, 
if the following three requirements hold:
\begin{enumerate}
\item \label{Jjkbjkdb1} $\mathfrak M\models q\leq\bar r$ and $q$ is  determined;
\item \label{Jjkbjkdb3} the mapping $(\alpha,\beta) \mapsto p(\alpha,\beta)$  is injective over $(A_{q}\setminus A_{\bar r}) \times (A_{r}\setminus A_{\bar r})$;
 \item \label{Jjkbjkdb4} $p(\alpha,\beta)>p(\alpha',\beta')$ for all $(\alpha,\beta)\in (A_{q}\setminus A_{\bar r}) \times (A_{r}\setminus A_{\bar r})$ and $(\alpha',\beta')\in [A_{r}]^2\cup[A_{q}]^2$,
\end{enumerate}
then $q\not\perp r$.
\end{itemize}
\end{lemma}
\begin{proof} Proceed by induction on $\zeta\le\sup(\eta)$ proving {\bf A} and {\bf B} simultaneously.
The case $\zeta = 0$ is immediate.
The case $\zeta=1$ is simple as well, but it may be instructive to consider it in detail.
So $c_0$ is a coloring in the ground model and all conditions are determined.
In effect, given $r\in\mathbb Q_1$, $r^{\mathfrak M}$ is a condition, as well.
It will be shown that $\bar{z} = \emptyset$ satisfies the conclusion.

Let $k$ be sufficiently large for $r$.
We know that $(r^{\mathfrak M})^\wedge (k, z)\in \mathfrak M\cap \mathbb Q_1$ for any $z\in [\mathfrak M\cap\omega_1]^{<\aleph_0}$.
Hence
{\bf A} is immediate.
To see that {\bf B} holds,
suppose that we are given  $z\in[\mathfrak M\cap\omega_1]^{<\aleph_0}$,
we let $\bar r := (r^\mathfrak{M})^\wedge (k, z)$,
and we are also given a condition $q\in \mathfrak M\cap\mathbb Q_1$ satisfying requirements (1)--(3) above.

To see that $q\not\perp r$, let $a:=a_{q,0}\cup a_{r,0}$,
$f := f_{q,0}\cup f_{r,0}$ and $w := w_{q,0}\cup w_{r,0}$.
It is immediate to see that $f$ and $w$ are functions, 
$A_r=a_{r,0}$, $A_q=a_{q,0}$ and $A_q\cap A_r=A_{\bar r}$.
We need to show that there exists a function $w^*$ extending $w$ for which $(a,f,w^*)$ is a legitimate condition.
For this, suppose that we are given $i,j<\omega$, $(\alpha,\beta), (\alpha',\beta')\in [a]^2$,
with $f(\alpha) = f(\beta) = i = f(\alpha') = f(\beta')$ and
$p(\alpha,\beta) = j = p(\alpha',\beta')$.
It must be shown that $c_0(\alpha,\beta) = c_0(\alpha',\beta')$.
There are two cases to consider:
\begin{description}
\item[Case I] If $(\alpha,\beta),(\alpha',\beta')\in [A_q]^2\cup [A_r]^2$, then since $w$ extends $w_{q,0}$ and $w_{r,0}$,
$c_0(\alpha,\beta) = w(i,j)=c_0(\alpha',\beta')$.
\item[Case II]  If $(\alpha,\beta)\in [a]^2\setminus([A_q]^2\cup [A_r]^2)$, then since $A_q\cap A_r=A_{\bar r}$ and $\alpha<\beta$,
we infer that $(\alpha,\beta)\in (A_{q}\setminus A_{{\bar r}}) \times (A_{r}\setminus A_{{\bar r}})$. So, by Clause~(3),
$(\alpha',\beta')\in [a]^2\setminus([A_q]^2\cup [A_r]^2)$, as well.
Then, likewise $(\alpha',\beta')\in (A_{q}\setminus A_{{\bar r}}) \times (A_{r}\setminus A_{{\bar r}})$. 
Altogether, by Clause~(2), $(\alpha,\beta)=(\alpha',\beta')$. In particular, 
$c_0(\alpha,\beta) = c_0(\alpha',\beta')$.
\end{description}

Next, assume that $\zeta\le\sup(\eta)$ and that {\bf A} and {\bf B} have been established for all $\xi<\zeta$.
If $\zeta$ is a limit, then the finite-support nature of the iteration
also establishes both {\bf A} and {\bf B} hold, so suppose that
$\zeta = \xi+1$.
The successor case in which
$\xi \notin \mathfrak M$ also follows directly from the induction
hypothesis by the definition of $q^\mathfrak{M}$,
so assume that $\xi \in \mathfrak M$. 

Let $r \in \mathbb Q_{{\zeta}}$ be  determined.
Let $\bar z\in[\mathfrak M\cap\omega_1]^{<\aleph_0}$ be given by the induction hypothesis with respect to $r\restriction \xi$.
In particular, 
for every $z\in[\mathfrak M\cap\omega_1]^{<\aleph_0}$ covering $\bar z$, and $k$ sufficiently large for
$r\restriction \xi$, 
$((r\restriction \xi)^\mathfrak{M})^\wedge (k, z)$ is in $\mathfrak{M}\cap \mathbb Q_{\xi}$ and is  determined.

To establish {\bf A},
note that, since $\xi\in \mathfrak M$, it follows that for any $z\in[\mathfrak M\cap\omega_1]^{<\aleph_0}$ covering $\bar z$, and $k$ sufficiently large for $r$ (in particular, sufficiently large for $r\restriction \xi $),
$s_{k,z}:=((r\restriction \xi)^\mathfrak{M})^\wedge (k,z)*(a_{r,\xi}\cap \mathfrak M,f_{r,\xi}\cap \mathfrak M,w_{r,\xi})$
is in $\mathfrak M$. 
It must also be shown that $s_{k,z}$ belongs to $\mathbb Q_{\zeta}$.
For this, it suffices to show that for all $i,j<\omega$,
$$((r\restriction \xi)^\mathfrak{M})^\wedge (k, z)\Vdash_{\mathbb Q_{\xi}}
``\forall (\alpha,\beta)\in [f_{r, \xi}^{-1}[\{i\}]\cap p^{-1}[\{j\}]\cap\mathfrak M]^2~\forcingname{c}_\xi(\alpha,\beta)=w_{r,\xi}(i,j)".$$

As $\forcingname{c}_{\xi}$ belongs to $\mathfrak M$, for each $(\alpha,\beta)\in [a_{r,\xi}\cap \mathfrak M]^2$ there
is a countable, maximal antichain
deciding $\forcingname{c}_{\xi}(\alpha,\beta)$ and
belonging to $\mathfrak M$ because; in other words, all possible decisions about the value of
$\forcingname{c}_{\xi}(\alpha,\beta)$ can be forced without leaving $\mathfrak M$.
So, if the above displayed assertion fails, then there must be some
$q^*\leq ((r\restriction \xi)^\mathfrak{M})^\wedge (k, z)$ in $\mathfrak M$ 
and $(\alpha,\beta)$ in $[f_{r,\xi}^{-1}[\{ i\}]\cap \mathfrak M]^2$ such that
$p(\alpha,\beta) = j$, but
$q^*\forces{\mathbb Q_{\xi}}{\forcingname{c}_{\xi}(\alpha,\beta) \neq w_{r,\xi}(i,j) }$.
Fix $k$ sufficiently large for $r$.
Then, under the assumption that  for any $z\in[\mathfrak M\cap\omega_1]^{<\aleph_0}$ covering $\bar z$,
$s_{k,z}$ does not belong to $\mathbb Q_{\zeta}$,
it is possible to construct recursively a sequence
$\{(z_n,q_{n}, i_n, j_n,(\alpha_n,\beta_n)\}_{n\in\omega}$ such that:
\begin{itemize}
\item $z_0 = \bar{z}$;
\item $q_{n}\leq ((r\restriction \xi)^\mathfrak{M})^\wedge (k, z_n)$ and $q_n$ is determined;
\item $(\alpha_n,\beta_n)\in [f_{r,\xi}^{-1}[\{ i_n\}]\cap p^{-1}[\{j_n\}]\cap \mathfrak M]^2$;
\item $q_{n}\forces{\mathbb Q_{\xi}}{\forcingname{c}_{\xi}(\alpha_n,\beta_n) \neq w_{r,\xi}(i_n,j_n)}$;
\item $z_{n+1}\supsetneq A_{q_n}$.
\end{itemize}
By making canonical choices (e.g., by consulting with $\lhd_\kappa$), this construction can be carried out in $\mathfrak M$.
Let $\epsilon:=\sup(\bigcup_{n\in\omega}A_{q_n})+1$. Define a function $h:\epsilon\to \omega$ via
$$h(\alpha) := \max\{k, p(\alpha',\beta')\mid (\alpha',\beta')\in [A_{q_{n+1}}]^2\text{ and }\alpha \in A_{q_{n+1}}\setminus A_{q_n}\},$$
and note that $h$ is in $\mathfrak M$.

Let $\gamma$ satisfy the conclusion of Lemma~\ref{jsdbtgy} for $h$.
As $h\in\mathfrak M$, $\gamma\in\mathfrak M$, so, since $b:=A_{r}\setminus \mathfrak M$ is an element of $[\omega_1\setminus\gamma]^{<\aleph_0}$,
  there exists $\Delta\in[\epsilon]^{<\aleph_0}$ such that:
  \begin{itemize}
  \item $p\restriction ((\epsilon\setminus\Delta)\times b)$ is injective;
  \item for all $\alpha\in\epsilon\setminus\Delta$  and $\beta\in b$, $h(\alpha)<p(\alpha,\beta)$.
  \end{itemize}

Fix a large enough $n<\omega$ such that $A_{q_{n+1}}\setminus A_{q_n}$ is disjoint from $\Delta$. 
Denote $\bar r:=((r\restriction \xi)^\mathfrak{M})^\wedge (k, z_{n+1})$.
As $z_{n+1}\supseteq A_{q_n}$, $(A_{q_{n+1}}\setminus A_{\bar r})\s(A_{q_{n+1}}\setminus A_{q_n})\s(\epsilon\setminus\Delta)$,
$(A_{r\restriction\xi}\setminus A_{\bar{r}}  )\s b$, and all of the following hold:
\begin{enumerate}
\item $\mathfrak M\models q_{n+1}\leq\bar r$ and $q$ is  determined;
\item the mapping $(\alpha,\beta) \mapsto p(\alpha,\beta)$  is injective over 
$(A_{q_{n+1}}\setminus  A_{\bar r})\times (A_{r\restriction\xi}\setminus A_{\bar{r}}  )$;
 \item $p(\alpha,\beta)>p(\alpha',\beta')$ 
for all  $(\alpha,\beta)\in (A_{q_{n+1}}\setminus  A_{\bar r})\times (A_{r\restriction\xi}\setminus A_{\bar{r}}  )$
and  $(\alpha',\beta')\in [A_{r\restriction\xi}]^2\cup [A_{q_{n+1}}]^2$.
\end{enumerate}

Then applying the induction hypothesis for {\bf B} yields that $q_{n+1}\not\perp (r\restriction \xi)$.
Pick a determined condition $q^*$ in $\mathbb Q_{\xi}$ simultaneously extending $q_{n+1}$ and $(r\restriction \xi)$. As $q^*\le q_{n+1}$, we infer that 
$$q^*\forces{\mathbb Q_{\xi}}{\forcingname{c}_{\xi}(\alpha_{n+1},\beta_{n+1}) \neq w_{r,\xi}(i_{n+1},j_{n+1})  }.$$ 
As $(\alpha_n,\beta_n)\in [f_{r,\xi}^{-1}[\{ i_n\}]\cap \mathfrak M]^2$,
$p(\alpha_n,\beta_n) = j_n$ and $q^*\le r\restriction\xi$,
we infer that
$$q^*\forces{\mathbb Q_{\xi}}{\forcingname{c}_{\xi}(\alpha_{n+1},\beta_{n+1}) =w_{r,\xi}(i_{n+1},j_{n+1})  }.$$ 
This is a contradiction. So {\bf A} does hold.

Next, let us establish {\bf B}.
Recall that we have a determined condition $r \in \mathbb Q_{{\zeta}}$
and $\bar z\in[\mathfrak M\cap\omega_1]^{<\aleph_0}$
satisfying that for every $z\in[\mathfrak M\cap\omega_1]^{<\aleph_0}$ covering $\bar z$, and $k$ sufficiently large for
$r\restriction \xi$, 
$((r\restriction \xi)^\mathfrak{M})^\wedge (k, z)$ is in $\mathfrak{M}\cap \mathbb Q_{\xi}$ and is  determined.
We have just established {\bf A}, proving that we may fix a finite $z^*$ with $\bar z\s z^*\s\mathfrak M\cap\omega_1$,
satisfying that for every $z\in[\mathfrak M\cap\omega_1]^{<\aleph_0}$ covering $z^*$, and every integer $k$ that is sufficiently large for $r$,
$(r^\mathfrak{M})^\wedge(k, z)$ is in $\mathfrak M\cap \mathbb Q_\zeta$ and is determined.

Now, fix arbitrary $z\in[\mathfrak M\cap\omega_1]^{<\aleph_0}$ covering $z^*$, 
an integer $k$ that is sufficiently large for $r$,
and a condition $q\in \mathfrak M\cap\mathbb Q_\zeta$.
Set $\bar r := (r^\mathfrak{M})^\wedge (k, z)$
and suppose that the requirements (1)--(3) of {\bf B} for $\bar r$ and $q$ hold.
In particular, they hold for $\bar r\restriction\xi$ and $q\restriction\xi$. That is:
\begin{itemize}
\item  $\mathfrak M\models q\restriction\xi\leq\bar r\restriction\xi$ and $q\restriction\xi$ is  determined;
\item the mapping $(\alpha,\beta) \mapsto p(\alpha,\beta)$  is injective over $(A_{q\restriction\xi}\setminus A_{\bar r\restriction\xi}) \times (A_{r\restriction\xi}\setminus A_{\bar r\restriction\xi})$;
 \item $p(\alpha,\beta)>p(\alpha',\beta')$ for all $(\alpha,\beta)\in (A_{q\restriction\xi}\setminus A_{\bar r\restriction\xi}) \times (A_{r\restriction\xi}\setminus A_{\bar r\restriction\xi})$ and $(\alpha',\beta')\in [A_{r\restriction\xi}]^2\cup[A_{q\restriction\xi}]^2$.
\end{itemize}

Now, as $z^*\supseteq\bar z$, we get from {\bf B} of the previous stage that $(q\restriction\xi)\not\perp(r\restriction\xi)$.
Pick a determined condition $q^*$ in $\mathbb Q_{\xi}$ simultaneously extending $(q\restriction\xi)$ and $(r\restriction \xi)$. 
Let $a := a_{q,\xi}\cup a_{r,\xi}$ , $f := f_{q,\xi}\cup f_{r,\xi}$ and $w:= w_{q,\xi}\cup w_{r,\xi}$.
It is immediate to see that $f$ and $w$ are functions, 
$A_r\supseteq a_{r,\xi}$, $A_q\supseteq a_{q,\xi}$ and $A_q\cap A_r=A_{\bar r}$.
To see that $q\not\perp r$, it suffices to prove that there exists $w^*\supseteq w$ such that $q^** (a,f,w^*)\in\mathbb Q_{\zeta}$.

For this, suppose that we are given $i,j<\omega$, $(\alpha,\beta), (\alpha',\beta')\in [a]^2$,
with $f(\alpha) = f(\beta) = i = f(\alpha') = f(\beta')$ and
$p(\alpha,\beta) = j = p(\alpha',\beta')$.
It must be shown that 
$q^*\forces{\mathbb Q_{\xi}}{\forcingname{c}_{\xi}(\alpha,\beta)=\forcingname{c}_{\xi}(\alpha',\beta')  }$.
There are two cases to consider:
\begin{description}
\item[Case I] If $(\alpha,\beta),(\alpha',\beta')\in [a_{q,\xi}]^2\cup [a_{r,\xi}]^2$, then since $w$ extends $w_{q,\xi}$ and $w_{r,\xi}$,
the conclusion follows from the fact that $q^*$ extends $q\restriction\xi$ and $r\restriction\xi$.
\item[Case II]  If $(\alpha,\beta)\in [a]^2\setminus([a_{q,\xi}]^2\cup [a_{r,\xi}]^2)$, then,
as seen earlier, requirements (2) and (3) imply that  $(\alpha,\beta)=(\alpha',\beta')$.
\end{description}

So, we are done.
\end{proof}

\begin{lemma}\label{LuzinPres}
Suppose:
\begin{itemize}
\item  $p:[\omega_1]^2\to \omega$ satisfies the conclusion of Lemma~\ref{jsdbtgy};
\item  $L=\{l_\delta\}_{\delta\in\omega_1}$ is a strongly Luzin subset of $2^\omega$;
\item  $\{{{\mathbb Q}}_\xi\}_{\xi\in \eta}$ is a $p$-iteration with $\eta>0$ a limit ordinal.
\end{itemize}
Then $\one\forces{ {{\mathbb Q}_\eta}}{{L} \text{ is strongly Luzin}}$.
\end{lemma}
\begin{proof}
Suppose not. Then it can be assumed that there is a $\mathbb Q_\eta$-name $\forcingname{T}$ 
and a positive integer $d$ such that for the set $\mathcal Q_d:=\bigcup_{n<\omega}(2^n)^d$ of the `rationals' of $(2^\omega)^d$:
\begin{itemize}
\item $\one\forces{{{\mathbb Q}_\eta}}{ \forcingname{T}\subseteq \mathcal Q_d\text{ is a closed nowhere dense tree}}$, and
\item $\one\forces{{{\mathbb Q}_\eta}}{[\dot T]\cap L^d\text{ contains an uncountable pairwise disjoint subfamily}}$.
\end{itemize}
It follows that for each $\gamma<\omega_1$, we may fix a determined condition $r_\gamma\in  \mathbb Q_\eta$
and a sequence $\langle \delta^\gamma_i\mid i<d\rangle$ of ordinals in $\omega_1\setminus\gamma$ 
such that $r_\gamma\forces{{{\mathbb Q_\eta}} }{ \vec{l_\gamma}:=\langle l_{\delta^\gamma_i}\mid i<d\rangle\text{ is a branch through }\forcingname{T} }$.
Pick an uncountable $\Gamma\s\omega_1$ along with $k<\omega$ which is sufficiently large for $r_\gamma$ for all $\gamma\in\Gamma$.
By possibly shrinking $\Gamma$ further, we may also assume that $\{ A_{r_\gamma}\mid \gamma\in\Gamma\}$ forms a $\Delta$-system with  root $\rho$,
and that $\langle\{ \delta_i^\gamma\mid i<d\}\mid \gamma\in \Gamma\rangle$ consists of pairwise disjoint sets.

Let $\mathfrak M$ be a structure good for the $p$-iteration $\{\mathbb Q_\xi\}_{\xi\in \eta}$,
with $\rho,\forcingname{T},\mathbb Q_\eta\in \mathfrak M$.

For each $\gamma\in\Gamma$, let $\bar z_\gamma$ be given by Lemma~\ref{Oawewlkn11} with respect to $r_\gamma$.
Fix an uncountable $\Gamma'\subseteq \Gamma$ and
some $\bar{z}\in[\omega_1\cap\mathfrak M]^{<\omega}$ such that $\bar z_\gamma=\bar z$ for all $\gamma\in\Gamma'$.
By possibly shrinking further, we may assume the existence of $q$ such that $(r_\gamma)^{\mathfrak M} = q$ for all $\gamma\in\Gamma'$.
In particular, for every $z\in[\mathfrak M\cap\omega_1]^{<\aleph_0}$ covering $\bar{z}$,
$q^\wedge (k, z)\in \mathfrak M\cap \mathbb Q_\eta$ is determined.
Let $\{\tau_n\}_{n\in\omega}$ enumerate the set $\mathcal Q_d$.
Construct recursively a sequence $\{ (z_n,{q}_n, t_n)\}_{n\in\omega}$ such that:
\begin{itemize}
\item $z_0 = \bar{z}\cup\rho$;
\item $q_{n}\leq q^\wedge (k, z_n)$ and $q_n$ is a determined condition lying in $\mathfrak M$;
\item\label{lakc9u9}  $\tau_n\s  t_n \in \mathcal Q_d$ with $q_n\forces{\mathbb Q_\eta}{t_n\notin \dot{T}}$;
\item $z_{n+1}\supsetneq A_{q_n}$.
\end{itemize}

Let $\epsilon:=\sup(\bigcup_{n\in\omega}A_{q_n})+1$. Define a function $h:\epsilon\to \omega$ via
$$h(\alpha) := \max\{k, p(\alpha',\beta')\mid (\alpha',\beta')\in [A_{q_{n+1}}]^2\text{ and }\alpha \in A_{q_{n+1}}\setminus A_{q_n}\}.$$

Recalling that $p$ was given by Lemma~\ref{jsdbtgy},
we now fix $\gamma^*<\omega_1$ satisfying that for every $b\in [\omega_1\setminus\gamma^*]^{<\aleph_0}$,
  there exists $\Delta\in[\epsilon]^{<\aleph_0}$ such that:
  \begin{itemize}
  \item[(\textrm{I})] $p\restriction ((\epsilon\setminus\Delta)\times b)$ is injective;
  \item[(\textrm{II})] for all $\alpha\in\epsilon\setminus\Delta$  and $\beta\in b$, $h(\alpha)<p(\alpha,\beta)$.
  \end{itemize}

Clearly, $\Gamma^*:=\{\gamma\in\Gamma'\mid \min(A_{r_\gamma}\setminus\rho)>\gamma^*\}$ is uncountable.
For each $n<\omega$, consider the open set $U_n:=\{ \vec{l}\in(2^\omega)^d\mid \bigwedge_{i<d}(t_n(i)\s\vec{l}(i))\}$.
Then $W := \bigcap_{j=0}^\infty\bigcup_{j=n}^\infty U_{n+1}$ is a dense $G_\delta$ set,
and hence $L^d\setminus W$ contains no uncountable pairwise disjoint subfamily.
Consequently, we may find some $\gamma\in \Gamma^*$ such that $\vec{l_\gamma}$ is in $W$.
Set $b:=A_{r_\gamma}\setminus\rho$
and then find $\Delta\in[\epsilon]^{<\aleph_0}$ satisfying (\textrm{I}) and (\textrm{II}).
Fix a large enough $j<\omega$ such that $A_{q_{n+1}}\setminus A_{q_n}$ is disjoint from $\Delta$ for all $n\ge j$.
As $\vec{l_\gamma}\in W$, we may now fix some $n\ge j$ such that $\vec{l_\gamma}\in U_{n+1}$.
Denote $\bar r:=(q^\mathfrak{M})^\wedge (k, z_{n+1})$.
Then $(A_{q_{n+1}}\setminus A_{\bar r})\s(A_{q_{n+1}}\setminus A_{q_n})\s(\epsilon\setminus\Delta)$,
$(A_{r_\gamma}\setminus A_{\bar{r}}  )\s b$, and all of the following hold:
\begin{enumerate}
\item $\mathfrak M\models q_{n+1}\leq\bar r$ and $q$ is  determined;
\item the mapping $(\alpha,\beta) \mapsto p(\alpha,\beta)$  is injective over 
$(A_{q_{n+1}}\setminus  A_{\bar r})\times (A_{r_\gamma}\setminus A_{\bar{r}}  )$;
 \item $p(\alpha,\beta)>p(\alpha',\beta')$ 
for all  $(\alpha,\beta)\in (A_{q_{n+1}}\setminus  A_{\bar r})\times (A_{r_\gamma}\setminus A_{\bar{r}}  )$
and  $(\alpha',\beta')\in [A_{r_\gamma}]^2\cup [A_{q_{n+1}}]^2$.
\end{enumerate}

Since $z_{n+1}\supseteq\bar z$ and $\bar z$ was given by Lemma~\ref{Oawewlkn11},
we may apply {\bf B} and infer that ${q}_{n+1}\not\perp r_\gamma$.
However, $q_{n+1}\forces{\mathbb Q_\eta}{t_{n+1}\notin \dot{T}}$
and $r_\gamma\forces{{{\mathbb Q_\eta}} }{ \vec{l_\gamma}\text{ is a branch through }\forcingname{T} }$,
contradicting the fact that $t_{n+1}\s \vec{l_\gamma}$.
\end{proof}

\begin{proof}[Proof of Theorem \ref{nonmkjbs64hb}] Start with a model $V$ of $\gch$ in which there exists a coherent Souslin tree (see \cite[Proposition~2.5 and Theorem~3.6]{paper22}).
Using $\ch$, fix a strongly Luzin set $L$ and a partition $p$ as in Lemma~\ref{jsdbtgy}.
Let $\mathbb {Q}_{\omega_2}$ be the corresponding $p$-iteration, using $\mathcal H_{\aleph_2}$ as our bookkeeping device of names of colorings $\forcingname c_\xi$.
The iteration satisfies Property~K, being a finite-support iteration of Property~K posets,
hence the coherent Souslin tree survives. 
In addition, the $ccc$ of the iteration implies that for every coloring $c:[\omega_1]^2\to \omega$ in the extension, there is a tail of $\xi\in\omega_2$ such that $c$ admits a $\mathbb Q_\xi$-name in $\mathcal H_{\aleph_2}$ of $V$.
So,  in $V^{\mathbb {Q}_{\omega_2}}$, all colorings $c:[\omega_1]^2\rightarrow\omega$ are $p$-special.
Finally, by Lemma~\ref{LuzinPres}, the strongly Luzin set $L$ survives.
\end{proof}

\section{Acknowledgments}

Kojman was partially supported by the Israel Science Foundation (grant agreement 665/20).
Rinot was partially supported by the Israel Science Foundation (grant agreement 2066/18)
and by the European Research Council (grant agreement ERC-2018-StG 802756).
Stepr\={a}ns was partially supported by NSERC of Canada.

\end{document}